\documentclass[12pt,a4paper]{article}%
\usepackage{amsfonts}
\usepackage{amsmath}
\usepackage[a4paper]{geometry}
\usepackage{fancyhdr}
\usepackage{amssymb}
\usepackage{color}
\usepackage[pdftex]{hyperref}
\usepackage{graphicx}
\usepackage{subfigure}%
\usepackage{url}
\usepackage{enumerate}
\providecommand{\U}[1]{\protect\rule{.1in}{.1in}}
\hypersetup{                                                  a4paper,                                                  breaklinks                                            }
\newtheorem{theorem}{Theorem}

\newtheorem{definition}[theorem]{Definition}

\newtheorem{lemma}[theorem]{Lemma}

\newtheorem{proposition}[theorem]{Proposition}
\newtheorem{remark}[theorem]{Remark}

\renewenvironment{abstract}
{\par\noindent\textbf{\abstractname.}\ \ignorespaces}
{\par\medskip}
\begin{document}

\title{\textbf{Pedal Curves of a Bicorn}}
\author{\textit{Thierry Dana-Picard, Moshe Hanau, Shmuel Krichevsky}\\ndp@jct.ac.il, mozs770.hanau@gmail.com, shmuelkrichevskey@gmail.com\\Department of Mathematics \\Jerusalem College of Technology\\Israel}
\date{\today}
\maketitle

\begin{abstract}
\textit{We explore pedal curves (a special case of a geometric locus) of a the plane curve called a bicorn. This curve has two cusps, which induce the existence of singularities of the pedal. We study the singularities of the pedals in each case. For this various presentations are used for the curves, parametric and symbolic presentations. Their properties and influence on the quality of the output is briefly discussed. In particular, the determination of points of self-intersection requires a parametric presentation; in this case a rational parametrization is the most efficient. Work is performed in an environment involving Dynamic Geometry, automated commands and a Computer Algebra System. Finally we describe a connection between our plane curves and some art, making the topic here suitable for STEAM education.}

\end{abstract}


%


%
\thispagestyle{fancy}

\section{Introduction}

Plane algebraic curves are a traditional topic, going back to the ancient Greeks. The vast literature includes  printed catalogs such as \cite{yates} and online databases. For example,  the \href{url{https://mathshistory.st-andrews.ac.uk/Curves}}{MacTutor} website proposes a list of Famous Curves Index, giving often a detailed history of the presented curve. The database \href{mathcurve.org}{Mathcurve} is much larger; for many curves it offers several constructions, and also animations. Another characteristic is the multiple representations for almost every curve: when possible, a cartesian equation, a rational parametrization, a trigonometric parametrization, a complex presentation, etc. Multiple representations of the same mathematical object and switching between them, are important to study this object in an efficient way. As a  whole, we refer to Duval's registers of representation \cite{duval}; in our work on plane curves, the used registers are mostly graphical, symbolic and numerical.  We switch often between these registers: some properties are better studied using a parametrization, among them trigonometric parametrization (if available) provide a plot of better quality than a rational parametrization (if available), and these are in general of better quality than with an implicit plot using a cartesian equation. Moreover singularities (such as cusps) are easier classified using parametrization, but sometimes computations using curvature are more efficient to discover cusps; see \cite{kiss curve}.   
   
We study plane curves in a technology-rich environment. First, we explore with a software for Dynamic Geometry (DGS), here GeoGebra\footnote{Freely downloadable from \url{http://www.geogebra.org}} and its companion GeoGebra-Discovery\footnote{Freely downloadable from \url{https://github.com/kovzol/geogebra-discovery}. As it is always under development, we recommend to check frequently for a new version.}.  Its dynamic features (dragging points, slider bars, etc.,) enable exploration, leading to conjectures, which can then be proven using a Computer Algebra System (CAS). Such a CAS, called Giac \cite{kovacs parisse}, is embedded into GeoGebra. It happens that a stringer software is needed, in which case we use Maple (currently Maple 2025). Networking between different technologies provides important results, the automated commands of GeoGebra-Discovery having a crucial role \cite{safety,thales}. In this paper, we use its different versions of the automated commands to determine geometric loci: 
\begin{itemize}
\item 4 versions of a \textbf{Locus} command, all numerical. The output is a plot in the graphical window, which receives a generic name, as any variable.
\item 3 versions of a \textbf{LocusEquation} command, which is symbolic. The output has 2 fully synchronized components: a plot in the graphical window and an equation in the algebraic window. It happens that this command does not work, in which case a popup appears with a message telling that some steps of the construction are not supported by the command. It is worth to try an alternative way for the construction.
\end{itemize}
\begin{remark}
\label{remark irrelevant components}
Even when both commands provide outputs, they may be different. \textbf{LocusEquation} computes the closure in Zariski topology of the actual geometric locus under study, and irrelevant components may appear, which cannot always be distinguished by algebraic means (such as factorization of polynomials). This distinction may be obtained by a further dynamic exploration with the DGS.  
\end{remark}      
Pedal curves are presented in the  \href{https://mathcurve.com/courbes2d.gb/podaire/podaire.shtml}{dedicated page of Mathcurve} with some examples. 
\begin{definition}
\label{def pedal curve}
Let $\mathcal{C}$ be a plane curve, which will be called the \emph{base curve},  and $P$ a point in the plane, which will be called the \emph{pole} or the  \emph{pedal point}.  The pedal of $\mathcal{C}$  with respect to $D$ is the locus of the feet of the lines passing by $P$ and perpendicular to the tangents to the curve $\mathcal{C}$.
\end{definition}

Pedal curves of conics are studied in \cite{pedal sextics octics ACTM 2025}, leading to the discovery and study of curves of higher degree and giving examples for Remark \ref{remark irrelevant components} (in Section 4 there); Figure \ref{fig pedals of conics} shows GeoGebra's screenshots for pedals of (a) a parabola and (b) an ellipse with respect to an external point.
 \begin{figure}[htb]
\centering
 \includegraphics[width=4.5cm]{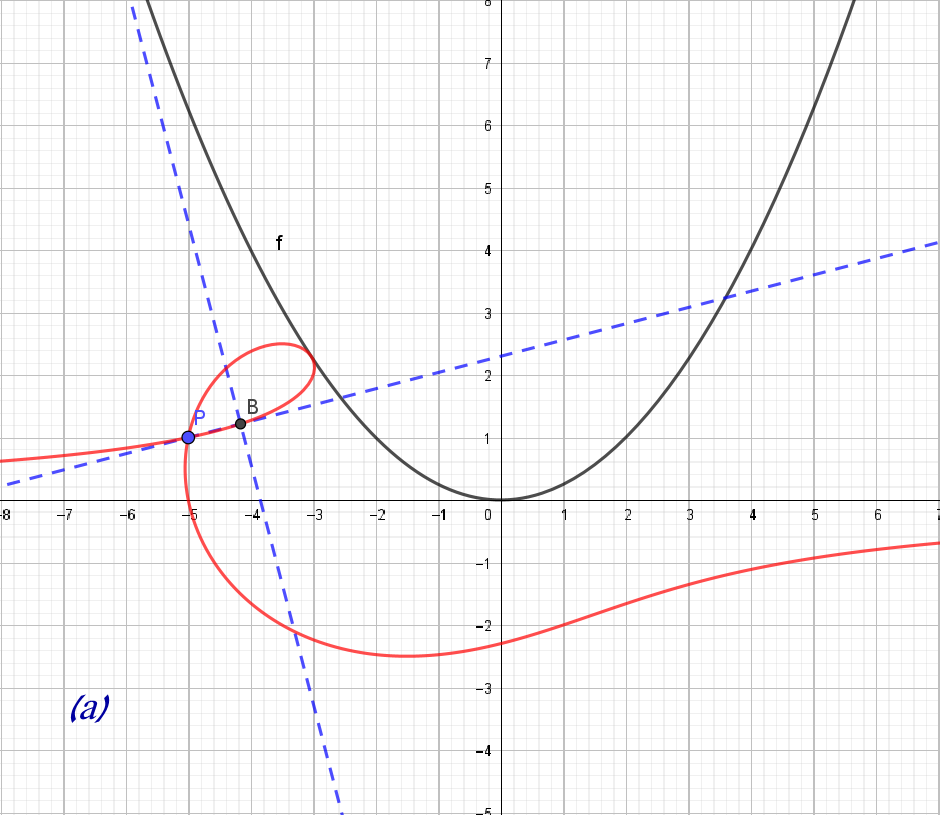}
\qquad
  \includegraphics[width=5.5cm]{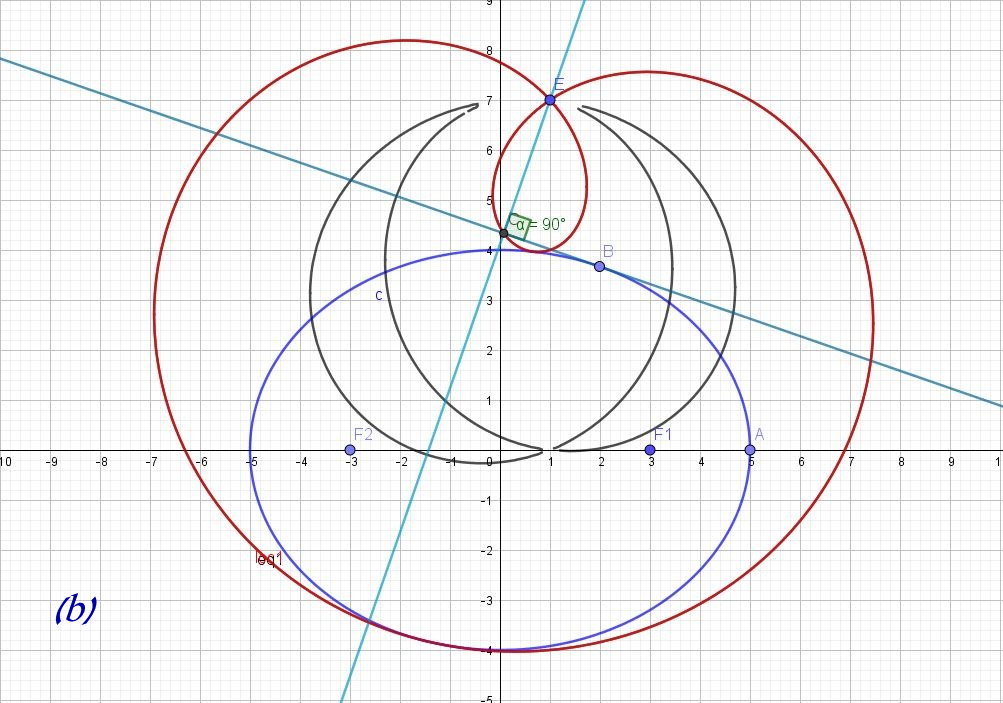}
\caption{Pedals of conics w.r.t. an external point}
\label{fig pedals of conics}
\end{figure}

In the present work, we explore properties of the pedal curves of a curve called the \textbf{bicorn} (or by its nickname, the "cocked hat"). We begin by deriving a parametrization of the pedal curve for some special cases of the location of the pole, and examine interesting properties of the singular points. 

\section{The study of singular points}
The study of singular points on a pedal curve is at the intersection of classical differential geometry, algebraic geometry, and applied kinematics. Take a base curve $\mathcal{C}$ and a pedal point $P$, the resulting pedal curve $\mathcal{C}_P$ often inherits or transforms singularities. It may also invents singularities, like cusps, nodes (self-intersection points, or tacnodes) that reveal hidden geometric properties of $\mathcal{C}_P$. We recall some of the important general properties.

In algebraic geometry, singular points are essential for determining the genus of an algebraic curve via the Pl\" ucker formula. The book \cite{sendra-winkler}) focuses extensively on the algorithmic determination of a curve's genus based on its singular structure. It  explicitly introduces the classical definition where the genus $g$ of an irreducible plane curve of degree $d$ with ordinary singularities $P_1, \dots, P_n$ of respective multiplicities $r_1, \dots, r_n$ is computed via the formula:  
\begin{equation}
\label{eq genus}
g = \frac{1}{2}\left[(d-1)(d-2) - \sum_{i=1}^n r_i(r_i-1)\right].
\end{equation}
See also Fulton's important book \cite{fulton}.

Even when the base curve $\mathcal{C}$ is smooth, the pedal transformation may create nodes or cusps, and increases the degree of the curve.
The singularities of a pedal curve do not appear at random; they correspond to critical intrinsic features of the base curve $C$:
\begin{itemize}
\item Inflexion points: An inflection point on the base curve $\mathcal{C}$ typically generates a cusp on its pedal curve $\mathcal{C}_P$; see \cite{weiss and giblin}.
\item The position of the pedal point: If the  pedal point $P$ lies on the base curve $\mathcal{C}$, the pedal curve $\mathcal{C}_P$ has a singular point at $P$. If $P$ is a focus of the base curve (for instance, of a conic), the pedal curve simplifies strongly (often becoming a circle or a straight line). See \cite{pedal sextics octics ACTM 2025}.
\end{itemize}

Singularities of pedal curves are studied for several applications, in physics, engineering, and optics. Among others, we have te following occurrences:
\begin{enumerate}
 \item Kinematics: movement of linkages, cams, and gears is modeled using plane curves \cite{kovacs linkages}. Understanding where cusps or double points occur allows engineers to design mechanisms that either avoid these disruptive physical boundaries or intentionally use them to achieve a sudden change in mechanical advantage.
\item Trajectory tracking: If a component is constrained to move such that it traces a path defined by a perpendicular projection (the essence of a pedal curve), the singularities represent, among other properties,  stationary points or reversals of motion. 
\item Optics and wavefronts (Caustics), the envelopes of light rays reflected or refracted by a curved surface. The singularities, specifically cusps, of these curves correspond to areas of intense light concentration, such as the bright lines you see at the bottom of a coffee mug on a sunny day. If the edge of your cup is circular, you see a cardioid on your coffee. Studying the singularities allows optical engineers to map and correct aberrations in lenses and parabolic reflectors.
\item Computer-Aided Geometric Design (CAGD): when modeling shapes using computer software, curves are frequently offset or transformed. A common problem in curve offsetting is the appearance of unwanted "loops" and self-intersections \cite{pedal sextics octics ACTM 2025,offsets of trifolium,offsets of cassini,kiss curve}. 
\item Ballistics and envelopes of Trajectories: in tracking families of parabolic or ballistic paths under gravitational and drag forces, the envelope of safety is often calculated (simple examples are described in \cite{safety}). A singularity on the boundary curve represents a point where multiple trajectories overlap or terminate, providing vital data regarding target coverage or impact zones.
\end{enumerate}
Summarizing, the singularities of a pedal curve translate differential properties of the base curve (like curvature variations and inflection points) into distinct, localized algebraic features (nodes and cusps) that can be easily categorized and calculated using computer algebra systems.

In this paper, we take a bicorn as base curve and study pedals for various positions of the pedal point. In each case, we analyze the singularities and derive general results. The different kinds of singular points request different methods. Therefore cusps are studied in Section \ref{section pole at the origin} (when the pole is at the origin) and attention is given to the differences between working with parametrization or with an implicit equation. Section \ref{pole on the y-axis} deals with a pole running on the $y-$axis, and Section \ref{section pole on the x-axis} with a pole running on the $x-$axis; the obtained shapes are different. Section \ref{section self-intersection} is devoted to points of self-intersection. A more general case is shown in Section \ref{general case pole}.
  
\section{The bicorn}
\label{sectino bicorn}
According to the local needs, we will use either of the following definitions.

\begin{definition}
\label{def cartesian bicorn}
The bicorn curve is the plane curve defined by the Cartesian equation:
\begin{equation}
\label{eq bicorn}
    y^2(a^2 - x^2) = (x^2 + 2ay - a^2)^2
\end{equation}
where $a$ is a real parameter.
\end{definition}
Parametrization of a plane curve is never unique. We chose here a trigonometric parametrization, which provides accurate plots (see Remark \ref{remark rational param} below).
\begin{definition}
\label{def trig param bicorn}
The bicorn curve is the plane curve defined by the following trigonometric parametrization:
\begin{equation}
\label{trig param bicorn}
\begin{cases} 
    x(u) = a \sin u\\ 
    y(u) = \frac{a \cos^2 u}{2 - \cos u}
\end {cases}
\end{equation}
where $u \in [0, 2 \pi )$.
\end{definition}

WLOG, as we are interested in the topology of the pedal curves,  for what follows we can take $a = 1$. After all, the parameter $a$ fixes only the scaling, not the topology. The bicorn is displayed in Figure \ref{fig bicorn}.
\begin{figure}[htb]
\begin{center}
\includegraphics[width=5cm]{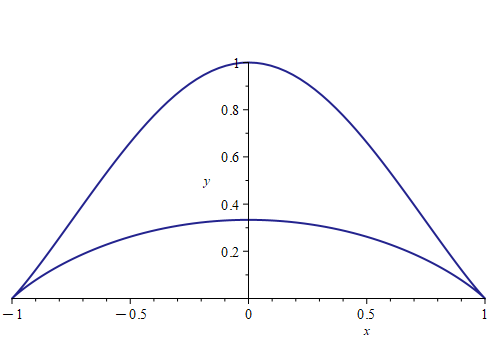}
\caption{The bicorn}
\label{fig bicorn}
\end{center}
\end{figure}

\begin{remark}
\label{remark rational param}
Using Maple's command \textbf{implicitize} (from the package \emph{algcurves}) derives Definition \ref{def cartesian bicorn} from Definition \ref{def trig param bicorn}. Some algebraic manipulation is needed to obtain exactly the same expression. In reverse direction, Maple's command \textbf{parametrization}  applied to  Definition \ref{def cartesian bicorn} will not yield the trigonometric parametrization, but a rational parametrization. With $a=1$, the obtained parametrization is as follows:
\begin{equation}
\label{rat param bicorn}
\begin{cases}
x = \frac{21 t^{2}+40 t+16}{29 t^{2}+40 t+16}\\
y= \frac{8 t^{2} \left(25 t^{2}+40 t+16\right)}{551 t^{4}+1688 t^{3}+2048 t^{2}+1152 t+256}
\end{cases}
\end{equation}

It is easy to show that the denominators have no real roots. Regarding the plot quality, the choice of the interval for the parameter has  big  influence, but a blank zone will always remain. A general treatment of the existence and the computation of rational parameterizations with a CAS approach is to be found in \cite{sendra-winkler}. In reverse direction,  the implicitization issue for trigonometric-hyperbolic parametrization is addressed in \cite{lastra sendra sendra}. 
 \end{remark}

Note that we can proceed as in \cite{safety,dynamic constructions,DP-Recio} (among others), and substitute 
\begin{equation}
\label{trig to rat}
\begin{cases}
\cos u =\frac{1-r^2}{1+r^2} \\
\sin u = \frac {2r}{1+r^2}
\end{cases}
\end{equation}
in Eq. (\ref{trig param bicorn}). This does not ensure that the obtained parametrization is identical to Eq. (\ref{rat param bicorn}).
In what follows, we use mostly the trigonometric parametrization, which ensures more accurate plots than rational parameterizations.

\section{Case 1: Pole at the Origin.}
\label{section pole at the origin}

In our first special case, we fix the pole to be at the origin. We will determine the pedal curve of the bicorn w.r.t. the origin in 2 settings: using the trigonometric parametrization, and using a rational parametrization.

\subsection{Work with the trigonometric parametrization}

To find the equation for the pedal curve, elementary Linear Algebra and Differential Geometry\footnote{The Maple code is provided as supplementary material.} is applied.  Starting with parametrization (\ref{trig param bicorn}), the following formulas for the derivative vector are obtained:
\begin{equation}
\label{derivative vector}
 \overset{\longrightarrow}{C'} : \quad 
 \begin{cases} 
 x'(u) = \cos u  \\ 
 y'(u) = \frac{\sin u \cos u(\cos u - 4)}{(2 - \cos u)^2} 
\end{cases}
\end{equation}
This vector vanishes for $u=\frac{\pi}{2}$ and $u=\frac{3\pi}{2}$,
corresponding to singular points of the bicorn. With iterated differentiations, we can prove that these points are cusps. For other values of the parameter, the general tangent line $T_u$ is given by the following cartesian equation:
\begin{equation}
\label{general tangent to bicorn - param}
 \left(-\cos u \; \sin u +4 \sin u \right) x +\left(\cos^2 u-4 \cos u+4\right) y +\left(2 \cos^2 u+\cos u-4\right)=0
\end{equation}

A direction vector for the normal to $T_u$ has coordinates
\begin{equation}
\label{normal vector}
\vec{N}(u)=
 \begin{cases} 
 x_N(u) = \frac{\sin u \cos u(\cos u - 4)}{(2 - \cos u)^2} \\ 
 y_N(u) = -\cos u  
\end{cases}
\end{equation}
whence an equation for the normal to $T_u$ through the origin is:
\begin{equation}
\label{normal through origin}
 \left(\cos^2 u-4 \cos u+4\right)x
 + \left(\cos u\sin u-4\sin u\right)y =0.
\end{equation}

The pedal curve is the geometric locus of the intersection points between the tangent $T_u$ to the original curve and its normal $N_u$ through the pole. Solving the system of equations yields a trigonometric parametrization of the pedal curve:
\begin{equation}
\label{trig param pedal wrt origin}
b(u)=
\begin{cases}
x_b(u) = \frac{\sin u \left(2 \cos^2 u+\cos u-4\right) \left(\cos u-4\right)}{9 \cos^2 u-40 \cos u+32},
\\
y_b(u) =-\frac{(2 \cos^2 u+\cos u-4) (\cos u-2)^2}{9 \cos^2 u-40 \cos u+32}
\end{cases}
\end{equation}
Figure \ref{fig pedal wrt origin} shows the bicorn and the pedal curve which has been determined.

\begin{figure}[htb]
\begin{center}
\includegraphics[width=5cm]{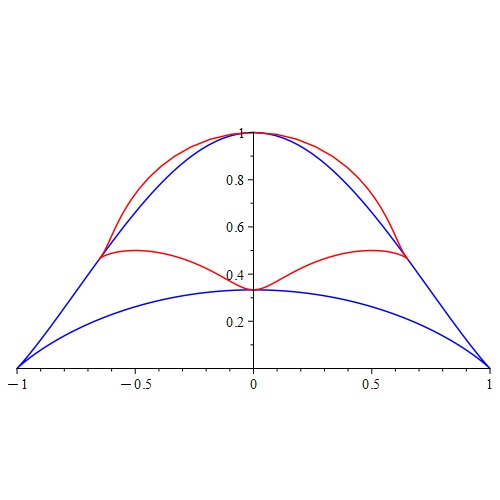}
\caption{The bicorn (in blue) and its pedal curve (in red) w.r.t. the origin}
\label{fig pedal wrt origin}
\end{center}
\end{figure}

\subsection{Work with a rational parametrization}
From Equations \eqref{general tangent to bicorn - param} and \eqref{trig to rat}, we derive a Cartesian equation for the tangent $T_u$ (at the point corresponding to the value $u$ of the parameter),
namely:
\begin{equation}
\label{general tangent to bicorn - cartesian}
2u(5u^2+3)x+(3u^2+1)^2y-(3u^4+12u^2+1)=0.
\end{equation}
The perpendicular $N_u$ to $T_u$ through the origin has equation
\begin{equation}
\label{normal to general tangent through origin}
(3u^2+1)^2x-2u(5u^2+3)y=0.
\end{equation}
The pedal curve of the bicorn w.r.t. the origin is the geometric locus of
points which solve the system of Equations
\eqref{general tangent to bicorn - cartesian}--\eqref{normal to general tangent through origin},
i.e. the curve with parametrization:
\begin{equation}
\label{pedl wrt origin - rational}
\begin{cases}
x = \frac{2u(5u^2+3)(3u^4+12u^2+1)}{(u^2+1)^2(81u^4+46u^2+1)} \\
y = \frac{(3u^2+1)^2(3u^4+12u^2+1)}{(u^2+1)^2(81u^4+46u^2+1)}
\end{cases}.
\end{equation}
Note that the denominator has no real roots. A parametric plot yields also Figure \ref{fig pedal wrt origin}.  

\subsection{Implicitization}
\subsubsection{With Gr\"obner bases algorithms} We start from the parametric parametrization \ref{trig param pedal wrt origin} and use the substitution 
given by Eq. (\ref{trig to rat}). The following Maple code gives the desired answer, where \emph{pedalpoint} denotes the trigonometric parametrization of the pedal curve.
\small
\begin{verbatim}
xpedalrat := simplify(subs(cos(u) = (-r^2 + 1)/(r^2 + 1), 
             subs(sin(u) = 2*r/(r^2 + 1), rhs(pedalpoint[1]))))
P1 := x*denom(xpedalrat) - numer(xpedalrat):
ypedalrat := simplify(subs(cos(u) = (-r^2 + 1)/(r^2 + 1), 
subs(sin(u) = 2*r/(r^2 + 1), rhs(pedalpoint[2]))))
P2 := y*denom(ypedalrat) - numer(ypedalrat):
J:=<P1,P2>:
JE:=EliminationIdeal(J,{x,y});
G:=Generators(JE)[1]:
\end{verbatim}
\normalsize
The output is the following octic polynomial:
\footnotesize
\begin{equation}
\label{cartesian eq pedal}
\begin{split}
G(x,y) &= 3 x^8 + 12 x^6 y^2 + 18 x^4 y^4 + 12 x^2 y^6 + 3 y^8 + 44 x^6 y + 132 x^4 y^3 + 132 x^2 y^5 + 44 y^7 + 18 x^6\\
&  + 165 x^4 y^2 + 276 x^2 y^4 + 129 y^6 - 252 x^4 y - 492 x^2 y^3 - 240 y^5 + 27 x^4 + 207 x^2 y^2 + 64 y^4
\end{split}
\end{equation}
\normalsize

\begin{proposition}
The pedal curve of the bicorn w.r.t. the origin is an octic curve whose implicit equation is $G(x,y)=0$, where $G$ is the polynomial defined in Eq. (\ref{cartesian eq pedal}).
\end{proposition}
\begin{remark}
The higher the degree of a 2-variable polynomial, the larger the set of shapes of  the corresponding curves.There exist complete catalogs of curves of degree 2,3,4, but only partial catalogs for higher degree. Here we have a new example of an octic, whose construction is totally different from the octics shown in \cite{thales,DP-Recio}.
\end{remark}

\section{The pole runs on the $y$-axis: the singular locus}
\label{pole on the y-axis}

We use parametrization \eqref{trig param bicorn} for the bicorn, and take the pole on the $y-$axis, namely $P(0,s)$. An equation of the tangent at a general point is 
\begin{equation}
\label{eq general tangent}
\left(-\cos u\sin u+4\sin u\right)x
+\left(\cos^2u-4\cos u+4\right)y
+\left(2\cos^2u+\cos u-4\right)=0.
\end{equation}
An equation of the normal to this tangent through the pole is
\begin{equation}
\label{eq general normal}
\left(\cos^2u-4\cos u+4\right)x
-\left(4-\cos u\right)\sin u \,y
+\left(4-\cos u\right)\sin u \,s=0.
\end{equation}

The pedal curve is determined by the common solutions of Equations (\ref{eq general tangent}) and (\ref{eq general normal}), i.e. by the following parametrization:
\begin{equation}
\label{param pedal pole on y-axis}
\begin{cases}
x = \frac{\left(\left(s+2\right) \cos^2 u+\left(1-4 s\right) \cos u+4 s-4\right) \sin u \left(-4+\cos u \right)}{9 \cos^2 u-40 \cos u+32} \\
y=-\frac{\left(-\cos^2 u+8 \cos u-16\right) \sin^2 u s+2 \cos^4 u-7 \cos^3 u+20 \cos u-16}{9 \cos^2 u-40 \cos u+32}
\end{cases}
\end{equation}

Figure \ref{fig:pedal-bicorn-y-axis} shows the bicorn and two of its pedal curves for poles on the $y$-axis.
\begin{figure}[htb]
\centering
\includegraphics[width=3.5cm]{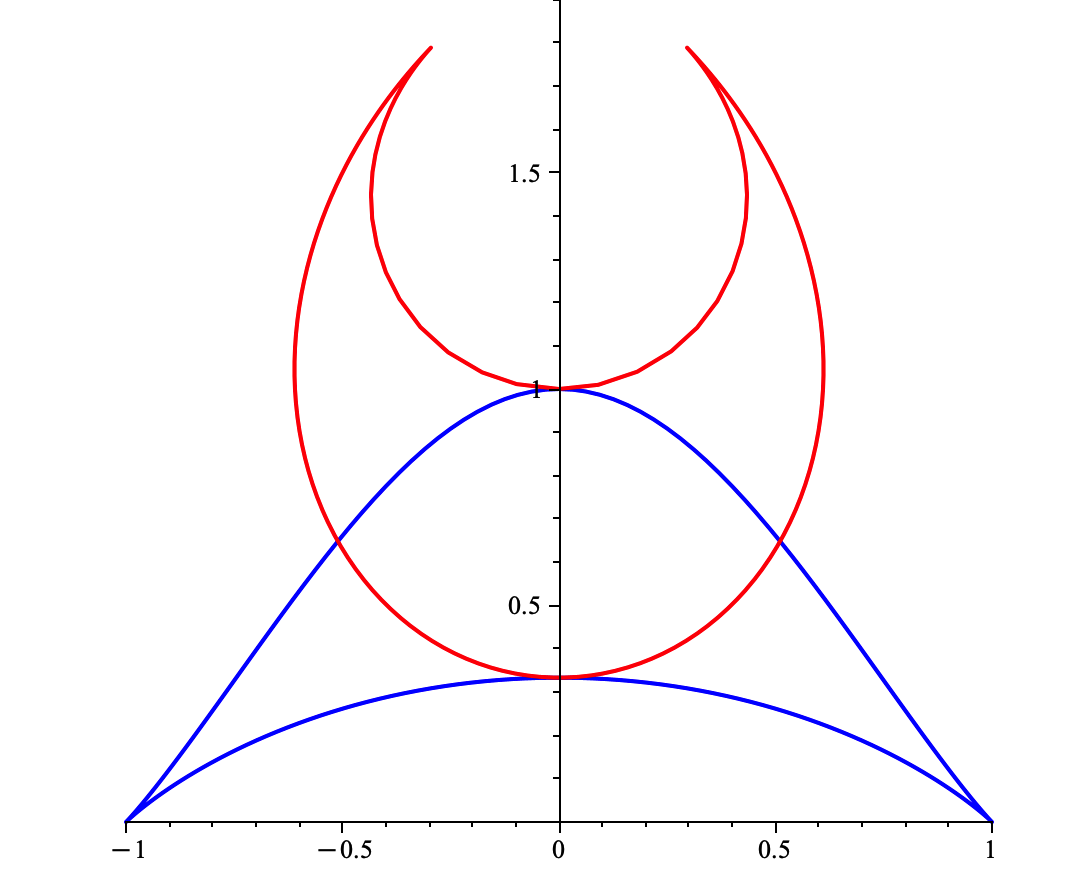}
\qquad
\includegraphics[width=4.5cm]{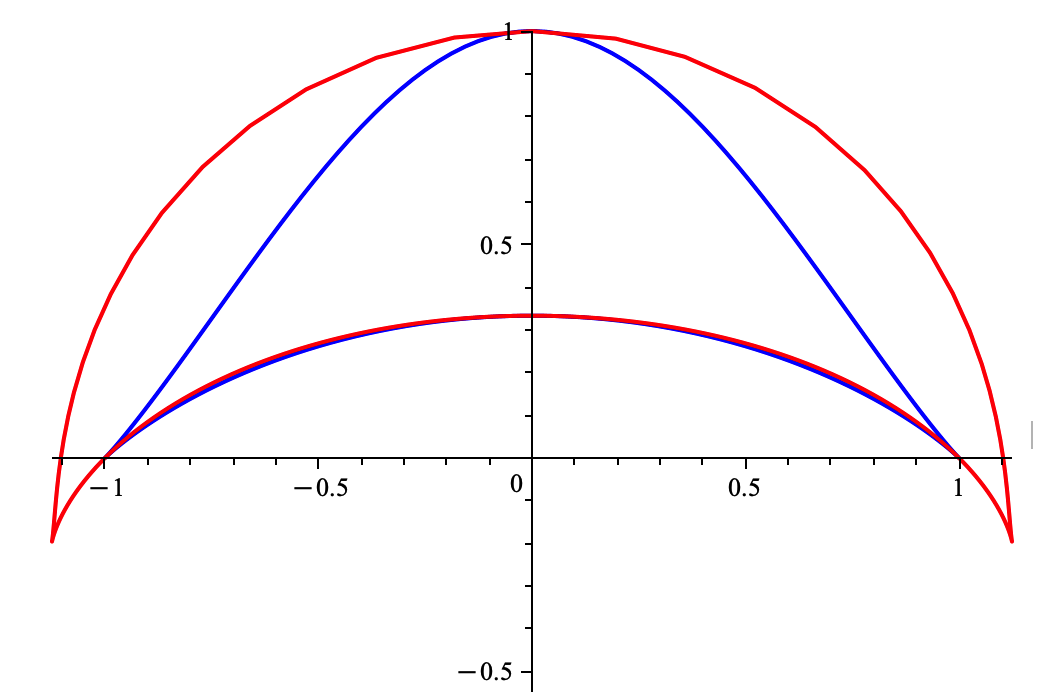}
\caption{The bicorn and its pedal curves for the poles $P=(0,2)$ and $P=(0,-1)$}
\label{fig:pedal-bicorn-y-axis}
\end{figure}

To determine the singular points of the pedal curve, it is more convenient to use a rational parametrization. With the substitution (\ref{trig to rat}), we obtain:
\begin{equation}
\label{general rat param pedal}
\begin{cases}
x(t) = -\frac{2 (5 t^{2}+3) t ( 9 s \,t^{4}-3 t^{4}+6 S \,t^{2}-12 t^{2}+s-1)}{(81 t^{4}+46 t^{2}+1) \; (t^{2}+1)^{2}},
\\
y (t)=\frac{27 t^{8}+100 s \,t^{6}+126 t^{6}+120 s \,t^{4}+84 t^{4}+36 s \,t^{2}+18 t^{2}+1}{(81 t^{4}+46 t^{2}+1) (t^{2}+1)^{2}}
\end{cases}
\end{equation}
Differentiating once yields
\begin{equation}
\label{1st derivative general rat param pedal}
\begin{cases}
x'(t)=\frac{\left(7290 s-2430\right) t^{12}+\left(1692 S-24864\right) t^{10}+\left(-6330 s-21390\right) t^{8}+\left(-3000 s-1968\right) t^{6}+\left(198 s+1542\right) t^{4}+\left(156 s-48\right) t^{2}-6 S+6}{\left(81 t^{4}+46 t^{2}+1\right)^{2} \left(t^{2}+1\right)^{3}}\\
y'(t)=-\frac{16200 \left(t^{2}+\frac{1}{3}\right) t \left(\left(s+\frac{17}{30}\right) t^{6}+\left(\frac{19 s}{15}-\frac{11}{90}\right) t^{4}+\left(\frac{23 s}{45}-\frac{7}{18}\right) t^{2}+\frac{s}{15}-\frac{1}{18}\right) \left(t^{2}-\frac{1}{5}\right)}{\left(81 t^{4}+46 t^{2}+1\right)^{2} \left(t^{2}+1\right)^{3}}
\end{cases}
\end{equation}
Maple's \textbf{solve} command provides only two real solutions, namely $t=\pm \sqrt{5}/5$. By substitution of $\sqrt{5}/5$ into (\ref{general rat param pedal}), we obtain a parametrization of one of the geometric loci of singular points
\begin{equation}
\label{param geom locus sing 1}
\begin{cases}
x = -\frac{125 \sqrt{5}\, \left(\frac{64 s}{25}-\frac{88}{25}\right)}{1512}\\
y = \frac{88}{189}+\frac{125 s}{189}
\end{cases}
\end{equation}
This can be easily implicitized:
\begin{equation}
\label{eq geom locus sing 1}
5x \sqrt{5}+8y-11=0.
\end{equation}
By the same method, substitution of  $-\sqrt{5}/5$ into (\ref{general rat param pedal}) leads to the following equation:
\begin{equation}
\label{eq geom locus sing 2}
5x \sqrt{5}-8y+11=0.
\end{equation}

\begin{proposition}
The singular locus of the pedal curve of the bicorn with pole on the $y$-axis is a subset of the union of two lines.
\end{proposition}

\begin{remark}
The same results can be obtained if, from the beginning, a rational parametrization is used for the bicorn.
\end{remark}

\begin{remark}
At this step, we have shown that the singular locus is contained in the union of the two lines. Proving that the entire union is obtained would require the
reverse construction: starting from a point on one of the lines and finding a pole on the $y$-axis for which this point appears as a singular point of the
corresponding pedal curve.
\end{remark}

\section{The pole runs on the $x$-axis}
\label{section pole on the x-axis}

The case where the pole lies on the $x-$axis is treated in the same way as the case where the pole lies on the $y-$axis. Take the pole
 $P=(v,0)$, use the same tangent line to the bicorn, and construct the line through $P$ orthogonal to this tangent.
Solving the corresponding system gives a parametric presentation of the pedal curve for this choice of the pole. The computations are similar to the previous section. Maybe surprising,  the same singular parameter values are obtained:
$ u=\arccos\left(\frac{2}{3}\right)$ and $u=2\pi-\arccos\left(\frac{2}{3}\right)$.
Consequently, the corresponding singular points lie on  the same two fixed lines, given by Equations (\ref{eq geom locus sing 1}) and (\ref{eq geom locus sing 2}).

\begin{remark}
The case $P=(V,0)$ is therefore not essentially different from the previous case. It provides another special instance of the general computation carried out in
the next section, where the pole is allowed to be an arbitrary point $P=(V,S)$.
\end{remark}

\section{General Case: Pole at $(r,s)$}
\label{general case pole}
We move to the general case where the pole is located at $(r,s)$. At a general point on the bicorn, the tangent has equation (\ref{eq general tangent}).
The equation for the orthogonal line to the tangent passing through the pole is:
\begin{equation}
\label{eq normal general pole}
\left(\cos^2u-4\cos u+4\right)(x-r)
-\left(4-\cos u\right)\sin u\,(y-s)=0.
\end{equation}

Solving the system \eqref{eq general tangent}--\eqref{eq normal general pole}
yields a general parametrization $b_{r,s}$, as follows:
\footnotesize
\begin{equation}
\label{general pedal general case}
b_{r,s}(u)=
\begin{cases}
x_{r,s}(u)= \dfrac{\sin u(\cos u-4)(2\cos^2u+\cos u-4)+r(\cos u-2)^4+s\sin u(\cos u-4)(\cos u-2)^2}
{9\cos^2u-40\cos u+32}\\
y_{r,s}(u)=\dfrac{s(\cos u-4)^2\sin^2u +r\sin u(\cos u-4)(\cos u-2)^2 -(\cos u-2)^2(2\cos^2u+\cos u-4)}
{9\cos^2u-40\cos u+32}
\end{cases}
\end{equation}
\normalsize
After differentiating, we find the tangent vector $\vec{b}'_{r,s}$.
Singular points appear when the two coordinates of this vector vanish simultaneously. Symbolic computation shows that the common vanishing is obtained here too when the following holds:
\begin{equation}
\label{singular values general pole}
    u = \arccos\left(\frac{2}{3}\right), \quad u = 2\pi - \arccos\left(\frac{2}{3}\right).
\end{equation}
Substituting these values of $u$ into $b_{r,s}(u)$ shows that the corresponding singular points lie on the two fixed lines already obtained in Equations \eqref{eq geom locus sing 1} and \eqref{eq geom locus sing 2}.

\begin{proposition}
For an arbitrary pole $P=(r,s)$, the singular locus of the pedal curve of the bicorn is contained in the union of the two fixed lines
$
5x\sqrt{5}+8y-11=0,
\qquad
5x\sqrt{5}-8y+11=0.
$
In particular, these lines are independent of the position of the pole.
\end{proposition}
\begin{remark}
The explicit expression of $\vec{b}'_{r,s}$ is rather long. Therefore, only
the resulting common parameter values are displayed here; the full symbolic
computation is included in the supplementary material.
\end{remark}

\underline{Geometric interpretation:} The position of the pole changes the shape of the pedal curve, but not the two fixed lines
on which its singular points may lie. Figure \ref{fig singular-lines} illustrates this phenomenon with $P=(1,2)$ and  $P=(-1,3/2)$ respectively.

\begin{figure}[h]
\centering
\includegraphics[width=5cm]{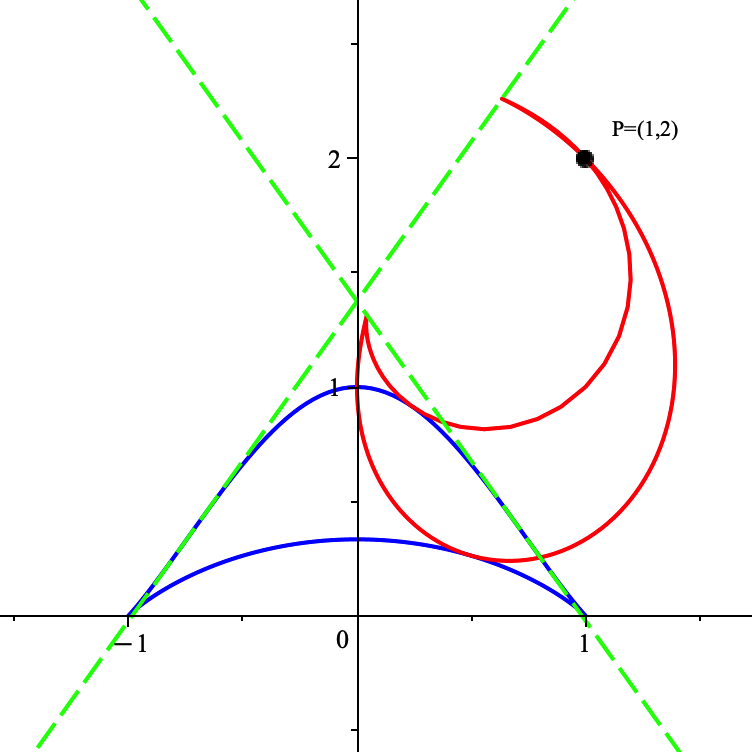}
\qquad \qquad
\includegraphics[width=5cm]{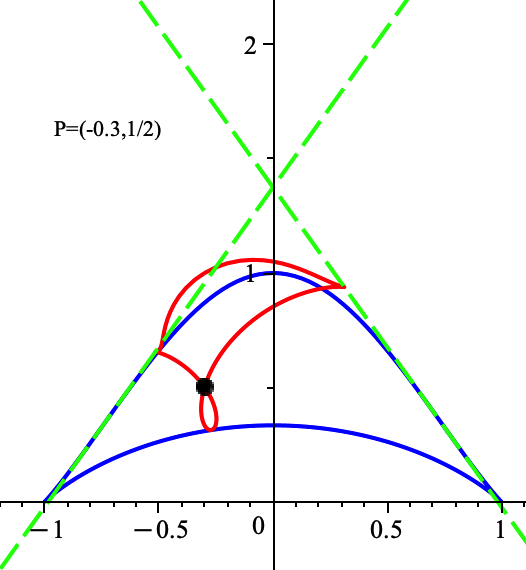}
\caption{The singular points of pedal curves of the bicorn lie on  two fixed tangent lines.}
\label{fig singular-lines}
\end{figure}

\begin{remark} 
A classroom exercise can be  to prove that every point on one of these lines is actually a singular point of a pedal curve of the bicorn. We leave this task to the reader.
\end{remark}

\section{Self-intersections of the pedal curve}
\label{section self-intersection}

After studying the singular locus of the pedal curve, it is natural to ask whether the pedal curve may also have self-intersections. Recall that a point of self-intersection corresponds to  two different values of the parameter, even if the tangent vector does not vanish at this point. In kinematics, the mobile point passes there at different times, eventually in different directions. 
By a non-trivial self-intersection we mean a point $Q$ for which there exist two distinct values of the parameter $m,n\in[0,2\pi)$ such that $ b_{r,s}(m)=b_{r,s}(n)=Q.$

Let $L_u$ denote the tangent line to the bicorn at the point corresponding to the parameter $u$. From Equation \eqref{eq general tangent}, an equation of this line can be
written in the slope-intercept form
\begin{equation}
\label{eq tangent line slope form}
L_u:\quad y= -\frac{(4-\cos u)\sin u}{(2-\cos u)^2}x + \frac{3-\cos u-\cos 2u}{(2-\cos u)^2}.
\end{equation}

We first prove a lemma which will be used in the study of self-intersections.

\begin{lemma}
\label{lemma equal tangent lines imply equal parameters}
Let $m,n\in[0,2\pi)$. If $L_m=L_n$, then $m=n$.
\end{lemma}

\noindent\textit{Proof.}
Assume that $L_m=L_n$. Comparing the two lines given by Equation \eqref{eq tangent line slope form}, we obtain equality of both their
slopes and their $y-$intercepts. The equality of the $y-$intercepts reads
\begin{equation}
\label{eq intercept equality self}
\frac{3-\cos m-\cos 2m}{(2-\cos m)^2}
=
\frac{3-\cos n-\cos 2n}{(2-\cos n)^2}.
\end{equation}
Using the classical identity 
$\cos 2t=2\cos^2 t-1$, 
Equation \eqref{eq intercept equality self} simplifies to
\begin{equation*}
(\cos n-\cos m) \; (3\cos n\cos m-4\cos n-4\cos m+4\ )=0.
\end{equation*}
Thus, we have two possible cases.

First, suppose that $\cos n=\cos m$. Since the slopes are also equal, we have:
\begin{equation*}
-\frac{(4-\cos m)\sin m}{(2-\cos m)^2}=-\frac{(4-\cos n)\sin n}{(2-\cos n)^2}.
\end{equation*}
Together with $\cos n=\cos m$, this implies that $\sin n=\sin m$. Hence, $m=n$ modulo $2\pi$. Since $m,n\in[0,2\pi)$, we obtain $m=n$.

The second possibility is $ 3\cos n\cos m-4\cos n-4\cos m+4=0$. Solving this equation for $\cos n$, we obtain
\begin{equation}
\label{cos n}
\cos n=\frac{4(\cos m-1)}{3\cos m-4}.
\end{equation}
Substitute this expression into the equality of the squared slopes. After simplification, we obtain:
\begin{equation*}
\frac{16(3\cos m-2)^3(\cos m-2)^5}{(3\cos m-4)^4}=0.
\end{equation*}
Therefore, $\cos m=\frac{2}{3}$.
Substituting this value back into Equation \eqref{cos n} yields
\begin{equation*}
\cos n=\frac{2}{3},
\end{equation*}
whence $\cos m=\cos n$. Returning again to the equality of the slopes, we obtain 
\begin{equation*}
-\frac{15\sin m}{8}=-\frac{15\sin n}{8},
\end{equation*} 
and therefore, $\sin m=\sin n$. 
Again, $m=n$ modulo $2\pi$. Since $m,n\in[0,2\pi)$, we conclude that $m=n$.
$\Box$

\begin{proposition}
If the pedal curve $b_{r,s}$ has a non-trivial self-intersection, then the self-intersection point is the pole $P=(r,s)$.
\end{proposition}

\noindent\textit{Proof.}
Assume, by contradiction, that the pedal curve has a non-trivial self-intersection at a point $Q\neq P$. Then there exist two different
parameters $m,n\in[0,2\pi)$, with $m\neq n$, such that
$b_{r,s}(m)=b_{r,s}(n)=Q.$

By the definition of the pedal curve, the point $Q=b_{r,s}(m)$ lies on the tangent line $L_m$, and the line $PQ$ is orthogonal to $L_m$. Similarly,
since $Q=b_{r,s}(n)$, the same point $Q$ lies on $L_n$, and the same line $PQ$ is orthogonal to $L_n$.

Since $Q\neq P$, the line $PQ$ is well-defined. Hence both $L_m$ and $L_n$ are orthogonal to the same line $PQ$. Therefore they are parallel.
Moreover, both lines pass through the common point $Q$. Two parallel lines with a common point coincide, and therefore
$L_m=L_n.$
By Lemma \ref{lemma equal tangent lines imply equal parameters}, it follows
that $m=n$, which contradicts the assumption that $m\neq n$.

Therefore, our assumption $Q\neq P$ is impossible. Hence every non-trivial self-intersection point of the pedal curve must be the pole $P=(r,s)$.

\section{Broadening horizons}
\subsection{Pedal curve as an envelope of circles}
\label{subsection envelope}
\begin{proposition}[\cite{lawrence}, p.48]
\label{prop pedal-envelope}
We use the notations of Definition \ref{def pedal curve}. The pedal curve of the curve $\mathcal{C}$ w.r.t. to the point $P$  is the envelope of the circles with diameter $PA$, when $A$ describes the curve $\mathcal{C}$.
\end{proposition}

We start a GeoGebra-Discovery session with the parametric equations (\ref{trig param bicorn}) for the  bicorn\footnote{It is available at \url{https://www.geogebra.org/m/kzfmue4g}}; a screenshot is  displayed in Figure \ref{fig pedal wrt origin}. The \textbf{Locus} command yielded a plot, but the \textbf{LocusEquation} command did not work. We computed a cartesian and a parametric presentation of the pedal curve in Section \ref{section pole at the origin}. 

The same cartesian equation can be obtained applying Proposition \ref{prop pedal-envelope}; we show here the process. We denote the pedal curve of the bicorn w.r.t. the origin by $\mathcal{C}_1$.
A  geometric construction of the circles is followed by applying the  \textbf{Envelope} command. The output is the green plot in Figure \ref{fig pedal wrt origin - envelope}, together with a polynomial equation $P(x,y)=0$ of degree 12, where
\scriptsize
\begin{equation}
\label{polyn eq pedal}
\begin{split}
P(x,y)= & 3 \; x^{12} + 3 \; y^{12} + 18 \; x^{2} \; y^{10} + 45 \; x^{4} \; y^{8} + 60 \; x^{6} \; y^{6} + 45 \; x^{8} \; y^{4} + 18 \; x^{10} \; y^{2} + 44 \; y^{11} \\
 &  + 220 \; x^{2} \; y^{9} + 440 \; x^{4} \; y^{7} + 440 \; x^{6} \; y^{5} + 220 \; x^{8} \; y^{3} + 44 \; x^{10} \; y + 15 \; x^{10} + 129 \; y^{10}  \\
 & + 531 \; x^{2} \; y^{8} + 834 \; x^{4} \; y^{6} + 606 \; x^{6} \; y^{4} + 189 \; x^{8} \; y^{2} - 240 \; y^{9} - 1016 \; x^{2} \; y^{7} - 1608 \; x^{4} \; y^{5} \\
 & - 1128 \; x^{6} \; y^{3} - 296 \; x^{8} \; y + 9 \; x^{8} + 64 \; y^{8} + 206 \; x^{2} \; y^{6} + 229 \; x^{4} \; y^{4} + 96 \; x^{6} \; y^{2} \\
 & + 240 \; x^{2} \; y^{5} + 492 \; x^{4} \; y^{3} + 252 \; x^{6} \; y - 27 \; x^{6} - 64 \; x^{2} \; y^{4} - 207 \; x^{4} \; y^{2}.
\end{split}
\end{equation}
\normalsize  
\begin{figure}[htb]
\begin{center}
\includegraphics[width=5cm]{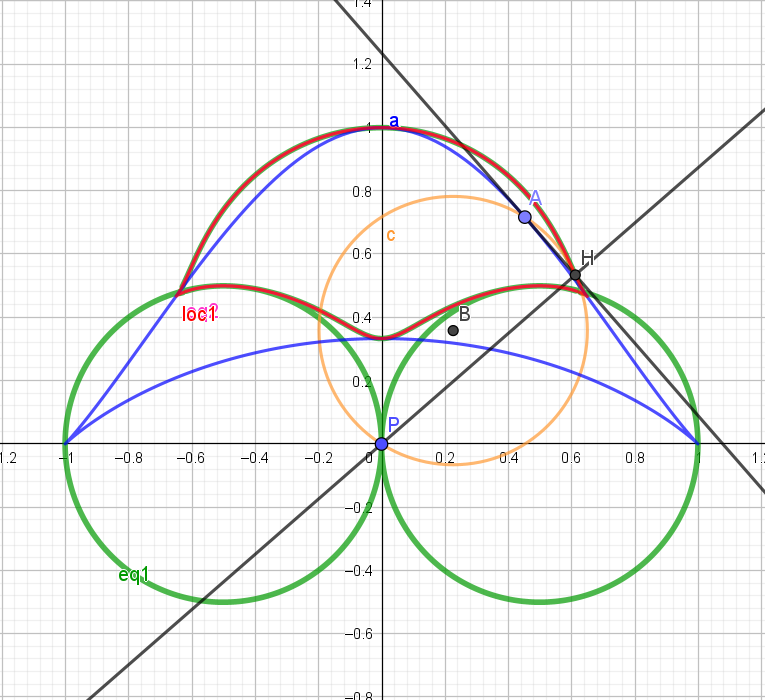}
\caption{The bicorn and its pedal curve as an envelope of circles}
\label{fig pedal wrt origin - envelope}
\end{center}
\end{figure}
This polynomial is reducible and has two quadratic factors $P_1(x,y)=x^{2} + y^{2} + x$ and  $P_2(x,y)=x^{2} + y^{2} - x$ (corresponding to the two circles appearing at the bottom in Figure \ref{fig pedal wrt origin - envelope}, and one irreducible factor of degree 8, namely the polynomial $G(x,y)$ in Equation \eqref{cartesian eq pedal}. Recall that the 2 irrelevant circles are a product of the computations in polynomial rings, and reflect a phenomenon in Zariski topology (such a discussion is beyond our scope here; it has been described for other situations in \cite{offsets of cassini,kiss curve}).
Note the ratio between the degrees: the bicorn is a quartic and its pedal curve w.r.t. the origin is an octic.

\subsection{Maths and Arts}

Numerous works are devoted to connections between math and visual arts. Among them, geometry is a central domain, as it enables to build activities suitable for any age. In \cite{saimon}, everyday artefacts from Tanzania are used as  a basis for learning geometry. The didactic interest is documented there. In another direction, closer to the present paper, Function Art described in \cite{arts} uses graphs of functions of a real variable are used to construct models of existing artefacts such as a kite or (with a small transition towards 3D) of bridges. An example of a ceramic floor displaying advanced curves (namely isoptics of Fermat curves) is given in \cite{fermat}.  

The basis curve in the present paper is a called a bicorn, or Napoleon's bicorn. It is easy to transform it into a 3D model of Napoleon's hat as a surface of revolution. The mathematical tools to use are the same which have been used for the bridges in \cite{arts}. 

The remark we wish to make is that numerous mathematical works around geometry, plane curves, surfaces, etc.   can be translated into arts. This is the case here. Figure \ref{fig lamp} shows a lamp on the wall in a public building form the area where one of the  authors lives. The shape fits the pedal curve of the bicorn described in Section \ref{pole on the y-axis}.

\begin{figure}[htb]
\begin{center}
\includegraphics[width=5cm]{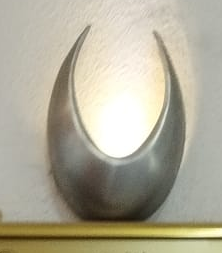}
\caption{A pedal curve of bicorn describes an artistic lamp}
\label{fig lamp}
\end{center}
\end{figure}

These last remarks explain why plane curves and related topics make this domain of study relevant to STEAM education.

\section{Conclusion}

The singular points depend only partially on the pole location, i.e. on the parameters $(r,s)$. They always lie on two fixed lines. By substituting the corresponding values of $u$ into the equation of the tangent line to the bicorn, we find the equations of these lines:
\begin{align}
    l_1 &: y = -\frac{15}{8}\sin\left(\arccos\frac{2}{3}\right)x + \frac{11}{8}
    =-\frac{5\sqrt{5}}{8}x+\frac{11}{8}
    \approx -1.40x + 1.38, \\
    l_2 &: y = \frac{15}{8}\sin\left(\arccos\frac{2}{3}\right)x + \frac{11}{8}
    =\frac{5\sqrt{5}}{8}x+\frac{11}{8}
    \approx 1.40x + 1.38.
\end{align}

In this paper, we studied the pedal curves of the bicorn with respect to different positions of the pole. In the first case, where the pole is located
at the origin, we obtained an explicit trigonometric parametrization of the pedal curve and, after implicitization by means of Gr\"obner bases, we showed
that the obtained curve is an octic curve.onstruction of a pedal curve. A characteristic if this process is that we needed to switch to an alternate presentation of a pedal curve, namely to view it as an envelope of circles. Duval \cite{duval,duval 2006} explains that mathematical objects cannot be grasped by hands, but are studied via different registers of representation. Regarding algebraic curves, the traditional registers are symbolic (cartesian equations, parametric presentations),numerical and graphical. Actually, a CAS or a DGS plots a curve when computing first a numerical representation, which is not displayed on the screen, only when explicitly requested to do it.

We considered then the case where the pole lies on the $y-$-axis. In this setting, the computation of the singular points showed that their geometric
locus is a subset of the union of the two fixed lines $l_1$ and $l_2$. This phenomenon persists in the general case: even when the pole is an
arbitrary point $P=(r,s)$, the parameter values which may yield singular points are independent of the pole location, and the corresponding singular
points remain constrained to the same two lines.

Finally, we examined possible self-intersections of the pedal curve. We proved that if two distinct parameter values define the same point of the pedal
curve, then this point must be the pole itself. Thus, apart from the possible appearance of the pole as a self-intersection point, the parametrization does
not produce additional self-intersections.

These results show that, although the pedal curves of the bicorn vary with the position of the pole, some of their geometric features are stable. In
particular, the two lines $l_1$ and $l_2$ play a central role in describing the singular behavior of the whole family of pedal curves.

The final section is devoted to study a pedal curve as an envelope. It is not common to see such a connection in textbooks. We should mention also that the notion of a geometric locus is frequently presented to high-school students, as it can be viewed in a totally geometric way, but envelopes of parametric families of curves are presented much later (if they are, see \cite{thom}) and requires tools from Calculus and Differential Geometry \cite{berger,bruce and giblin}.
Only when both topics have been acquired, the connection between them (and more, see \cite{lawrence}) can be explored.

\section*{Supplementary material}
\begin{enumerate}
\item Maple code for Section \ref{section pole at the origin}.
\item Maple code for the computation of the singular locus when the pole lies on the $y$-axis.
\item Maple code for the general case $P=(r,s)$.
\item Maple code for the verification of Lemma \ref{lemma equal tangent lines imply equal parameters}.
\end{enumerate}

\end{document}